\documentclass[12pt,a4paper]{article}
\usepackage{amsfonts,color,comment}
\usepackage{amsthm, thmtools, amsmath, amscd, amssymb}
\usepackage[utf8]{inputenc}
\usepackage{t1enc}
\usepackage[mathscr]{eucal}
\usepackage{indentfirst}
\usepackage{graphicx}
\usepackage{graphics}
\usepackage{pict2e}
\usepackage{epic}
\usepackage{url}
\usepackage{epstopdf}
\usepackage[english]{babel}
\usepackage{hyperref}
\usepackage{pgf,tikz}
\usetikzlibrary{arrows}
\definecolor{zzttqq}{rgb}{0,0,0.1}
\definecolor{zzttaa}{rgb}{0.6,0.2,0}
\usepackage{hhline}
\usepackage{enumitem}
\usepackage{tikz}
\usetikzlibrary{fit}
\usetikzlibrary{calc, shapes, positioning}
\usetikzlibrary{patterns.meta}
\usepackage{mathrsfs,pgf}
\usepackage{amsthm}
\usepackage{soul}

\usepackage[margin=2.5cm]{geometry}

\usepackage{fancyhdr}

\theoremstyle{definition}

\theoremstyle{plain}
\newtheorem{Th}{Theorem}
\newtheorem{lemma}[Th]{Lemma}
\newtheorem{cor}[Th]{Corollary}
\newtheorem{prop}[Th]{Proposition}
\newtheorem{obs}[Th]{Observation}
\newtheorem{quest}{Question}
 
\theoremstyle{remark}

 \theoremstyle{definition}

\newtheorem{Def}[Th]{Definition}

\newcommand{\RR}{{\mathbb R}}

\def\RR{\mathbb{R}}

\def\NN{\mathcal{N}}

\def\HH{\mathcal{H}}
\def\GG{\mathcal{G}}
\def\FF{\mathcal{F}}

\mathcode`l="8000
\begingroup
\makeatletter
\lccode`\~=`\l
\DeclareMathSymbol{\lsb@l}{\mathalpha}{letters}{`l}
\lowercase{\gdef~{\ifnum\the\mathgroup=\m@ne \ell \else \lsb@l \fi}}
\endgroup

\title{Stabbing non-piercing sets and face lengths in large girth plane graphs}

\author{
D\"om\"ot\"or P\'alv\"olgyi\thanks{ELTE Eötvös Loránd University and HUN-REN Alfréd Rényi Institute of Mathematics, Budapest. Supported by the NRDI EXCELLENCE-24 grant no.~151504 Combinatorics and Geometry and by the ERC Advanced Grant no.~101054936 ERMiD.} 
\and
Krist\'of Z\'olomy \thanks{ELTE Eötvös Loránd University, Budapest. Supported by the Thematic Excellence Program TKP2021-NKTA-62 of the National Research, Development and Innovation Office.
}
}

\makeatletter

\begin{document}

\thispagestyle{empty}
\maketitle

\begin{abstract}
    We show that a non-piercing family of connected planar sets with bounded independence number can be stabbed with a constant number of points.
    As a consequence, we answer a question of Axenovich, Kie{\ss}le and Sagdeev about the largest possible face length of an edge-maximal plane graph with girth at least $\ell$.\\

\end{abstract}

\section{Introduction}

Define a region as a connected planar compact set whose boundary consists of a finite number of disjoint Jordan curves. 
One of these curves is the outer boundary of the region, while the rest cut out ``holes''.
A family $\FF$ of regions is in general position if for any two regions from $\FF$ their boundaries intersect in finitely many points.\footnote{This is a technical condition that was introduced in \cite{RamanRay} and makes many of the arguments simpler, though most often could be omitted. We need to assume it as we will use a result from \cite{RamanRay} which was proved under this assumption, although we might get rid of it due to a very recent result \cite{RS25} which still needs to be verified.}
Such a family $\FF$ is \emph{non-piercing} if $F \setminus G$ is connected for any two regions $F,G\in \FF$.
For example, the family of all disks is non-piercing and, more generally, so is a pseudo-disk family, defined as a family of simply connected regions whose boundaries intersect pairwise at most twice; these include families formed by homothetic\footnote{A homothetic copy is a copy that is translated and scaled by a positive scalar factor.} copies of a fixed convex set.
However, a family of axis-parallel rectangles in general position is not necessarily non-piercing, as two rectangles can cross each other without any of them containing a vertex of the other.

For a family $\FF$, the independence number $\nu(\FF)$ is the size of the largest subfamily of pairwise disjoint sets, that is, the smallest number such that among any $\nu(\FF)+1$ sets there are two that intersect.
The piercing number $\tau(\FF)$ is the least number of points that pierce $\FF$, that is, the size of the smallest point set that intersects every set from $\FF$.
Our main result is the following.

\begin{Th}\label{thm:pierce}
    There is a function $f$ such that if $\FF$ is a family of non-piercing regions, then $\tau(\FF)\le f(\nu(\FF))$.
\end{Th}

Note that this implies that the disjointness graph of a family of non-piercing regions is $\chi$-bounded.\footnote{For the definition and a survey of $\chi$-boundedness, see \cite{SS20}.}
For results similar to Theorem \ref{thm:pierce} about homothetic copies of a fixed convex set with explicit bounds, see \cite{DJ12}.

In particular, if $\FF$ is a family of pairwise intersecting, non-piercing regions, then there exists an absolute constant $T$ such that $\FF$ can be pierced with at most $T$ points.
This result was proved earlier for pseudo-disk families \cite{ANPPSS} with different methods, using sweepings, which do not generalize to non-piercing families.
Instead, our proof, which can be found in Section \ref{sec:pierce}, uses the standard machinery developed to prove $(p,q)$-theorems.
We do not know the best possible value for $T$.
In case of pairwise intersecting disks, this is a well-studied problem where we know that the optimal value is four, for which by now there are several different proofs, see \cite{Danzer,Stacho,CKM}.
It is entirely possible that the answer in case of pairwise intersecting non-piercing regions is also four but at the moment this is not known even for pseudo-disks.

As an application of Theorem \ref{thm:pierce}, we answer a recent question of Axenovich, Kie{\ss}le and Sagdeev \cite{AKS} about the largest possible face length of an edge-maximal plane graph with girth at least $\ell$, which was our main motivation.

\begin{Th}\label{thm:faces}
    Suppose that $G$ is a plane graph with girth at least $\ell$, and that $G$ is edge-maximal with regards to these two properties. Then the length of any facial cycle of $G$ is at most $K \ell$ for some absolute constant $K$. 
\end{Th}

The exact statement of the problem and the proof can be found in Section \ref{sec:faces}.

In Section \ref{sec:facesgen}, we show a generalization of Theorem \ref{thm:faces} for graphs drawn on other orientable surfaces:

\begin{restatable}{Th}{facesgen}\label{thm:facesgen}
    For every genus $g$, there exists a constant $c_g$ such that for every edge-maximal graph $G$ with girth $\ell$ drawn on an orientable surface with genus $g$, the maximal face length of $G$ is bounded by $c_g\cdot \ell$.
\end{restatable}

\noindent
\emph{Remarks.} \;

Since the first version of our paper appeared on arXiv, Theorem \ref{thm:pierce} was improved by Huang, Qi, Rong and Xu \cite{HQRX25}, who showed that $\tau(\FF)\le c \nu(\FF)^9$ for some constant $c$.
They used the notion of dyadic VC-dimension and an old result of Ding, Seymour, and Winkler \cite{DSW} to obtain better bounds.

Later the bound was further improved to the optimal $\tau(\FF) = O(\nu(\FF))$ by Keller and Smorodinsky
\cite{KellerSmorodinsky25} who used a recent result of Dalal, Gangopadhyay, Raman, and Ray \cite{DGRR} that generalized a result of Pinchasi \cite{rom} from pseudo-disks to non-piercing regions.
This result is that in every family $\FF$ of non-piercing regions there is an  $F\in\FF$ which is ``small'' in the sense that there cannot be $157$ pairwise disjoint members of $\FF$ that intersect $F$.

In the new version of the paper \cite{AKS} (version 2 on arXiv), Axenovich, Kie{\ss}le, Sagdeev, and Zhukovskii also proved their original question, giving a concrete upper bound of $8\ell-13$, thus improving the bound of our Theorem \ref{thm:faces}.
They also raised the problem for surfaces of higher genus $g$, for which they first obtained a polynomial, later an $O_g(\ell^2)$ bound.
This inspired us to include Theorem \ref{thm:facesgen}, which improves this to $O_g(\ell)$.
In the third version of their paper \cite{AKS} (version 3 on arXiv), they also proved a bound of $O_\ell(g)$.
They also raised the problem whether a bound of $O(g\ell)$ holds.

\section{Proof of Theorem \ref{thm:pierce}}\label{sec:pierce}

In this section we present the proof of Theorem \ref{thm:pierce}.
We start with some definitions, related to $(p,q)$-theorems; for a complete survey of such results, see \cite{handbook}.

Following Hadwiger and Debrunner \cite{HD57}, we say that a family $\GG$ has the $(p,q)$-property, if for every subfamily $\GG \subset \FF$ with $|\GG|=p$, there exists a subsubfamily $\HH \subset  \GG$ of size $q$ with a non-empty intersection $\cap \HH \neq \emptyset$. In other words, from every $p$ sets from $\FF$, some $q$ intersect.
It was shown by Alon and Kleitman \cite{AK92} that if a family $\FF$ of compact convex sets in $\RR^d$ satisfies the $(p,q)$-property for any $p\ge q\ge d+1$, then $\FF$ can be pierced with $T(p,q)$ points.
Later, this result was extended from convex sets to several other families.
For the version that we need, we need to define the Vapnik-Chervonenkis dimension.
The VC-dimension of a family $\FF$ is the largest $d$ for which exists a set of $d$ elements, $X$, such that for every subset $Y\subset X$ there exists a set $F\subset \FF$ such that $X\cap F=Y$.
The dual VC-dimension $d^*$ of $\FF$ is the VC-dimension of the dual family $\FF^*$, in which the roles of elements and sets are swapped, with the containment relation reversed.
It is well-known and easy to see that $d^*\le 2^d$. 
Matou\v sek \cite{Matousek} showed that bounded VC-dimension and an appropriate $(p,q)$-property imply the existence of a small hitting set.

\begin{Th}[Matou\v sek \cite{Matousek}]\label{pqthm}
    If the dual VC-dimension of $\FF$ is at most $q-1$, and $\FF$ satisfies the $(p,q)$-property for some $p \geq q$, then the sets of $\FF$ can be hit with at most $T$ points, where $T$ is a constant depending on $p$ and $q$.
\end{Th}

Therefore, in order to prove Theorem \ref{thm:pierce}, it would be sufficient to show that non-piercing regions have bounded (dual) VC-dimension, and if a family $\FF$ of non-piercing regions has bounded independence number $\nu(\mathcal{F})$, then they also satisfy the $(p,q)$-property for some large enough $q$.

\begin{lemma}\label{lem:VCdim4}
If $\FF$ is a non-piercing family of regions, then the VC-dimension and the dual VC-dimension of $\FF$ are at most $4$. 
\end{lemma}

 Our proof is somewhat similar to \cite{BPR} where pseudo-disks were considered.

\begin{proof}
    Suppose for a contradiction that there exists a planar point set $\{x_1,\ldots,x_5\}$ shattered by $\FF$.
    This implies that for each pair of points $x_i,x_j$, there exists $F_{i,j} \in \FF$ which contains $x_i$ and $x_j$ but does not contain any of the other three points.
    Since $F_{i,j}$ is connected, there exists a simple curve $\gamma_{i,j}\subset F_{i,j}$ whose ends are $x_i$ and $x_j$.
    Let us fix one such curve for each of the ten pairs $(i,j)$ such that any two curves intersect a finite number of times.
    These ten curves form a planar embedding of the complete graph $K_5$ with vertices $\{x_1,\ldots,x_5\}$, therefore, by the strong Hanani-Tutte theorem \cite{Tutte}, there are two independent edges, without loss of generality, $\gamma_{1,2}$ and $\gamma_{3,4}$, that cross an odd number of times.
    As $F_{1,2}\setminus F_{3,4}$ is connected because of the non-piercing property, we can extend $\gamma_{1,2}$ to a Jordan-curve $\bar\gamma_{1,2}\subset F_{1,2}$ that still crosses $\gamma_{3,4}\subset F_{3,4}$ an odd number of times.
    Similarly, we can extend $\gamma_{3,4}$ to a Jordan-curve $\bar\gamma_{3,4}\subset F_{3,4}$ that still crosses $\bar\gamma_{1,2}$ an odd number of times.
    But then these two closed curves would violate the Jordan curve theorem.

    The argument for the dual VC-dimension is similar.
    Take five sets, $F_1,\ldots,F_5\in\FF$ that are shattered.
    Take five points, $\{x_1,\ldots,x_5\}$, such that $x_i\in F_i \setminus \cup_{k\ne i} F_k$ for each $i$ and ten points such that $y_{i,j}\in (F_i\cap F_j) \setminus \cup_{k\ne i,j} F_k$ for each $i<j$.
    Take simple curves $\gamma_{i,j}\subset F_i$ whose ends are $x_i$ and $x_{i,j}$ and $\gamma_{j,i}\subset F_j$ whose ends are $x_j$ and $x_{i,j}$ for each $i<j$.
    The ten concatenations $\gamma_{i,j}\gamma_{j,i}$ of the curves form a planar embedding of the complete graph $K_5$  with vertices $\{x_1,\ldots,x_5\}$, therefore, by the strong Hanani-Tutte theorem, there are two independent edges that cross an odd number of times.
    As each edge is the concatenation of two curves, $\gamma_{i,j}$ and $\gamma_{j,i}$, there are also two curves, without loss of generality, $\gamma_{1,2}\subset F_1$ and $\gamma_{3,4}\subset F_3$, that cross an odd number of times.
    From here the proof is the same as before.
    As $F_{1}\setminus F_{3}$ is connected because of the non-piercing property, we can extend $\gamma_{1,2}$ to a Jordan-curve $\bar\gamma_{1,2}\subset F_{1}$ that still crosses $\gamma_{3,4}\subset F_{3}$ an odd number of times.
    Similarly, we can extend $\gamma_{3,4}$ to a Jordan-curve $\bar\gamma_{3,4}\subset F_{3}$ that still crosses $\bar\gamma_{1,2}$ an odd number of times.
    But then these two closed curves would violate the Jordan curve theorem.
\end{proof}

Now we only need to show that a $(p,q)$-property holds for some $p\ge q\ge 5$ in every family $\FF$ of non-piercing regions with bounded independence number $\nu(\FF)$.
We first make some definitions.

For a family $\FF$ and collection of elements $P$, define the dual intersection hypergraph $\HH(\FF, P)$ such that the vertices correspond to members of $\FF$, while hyperedges correspond to elements of $P$ such that for every $p \in P$ the vertex set $H_p=\{F\in \FF: p \in F\}$ forms a hyperedge.
The Delaunay graph $\mathcal{D(\FF)}$ is the subgraph of $\HH(\FF, P)$ that contains only the hyperedges with exactly two vertices, that is, a pair of vertices corresponding to the sets $F,G \in \FF$ are connected by an edge if there is an element $p \in F\cap G$ which is not contained in any other member of $\FF$.
Raman and Ray proved (in a much more general form) that if $\FF$ is a family of non-piercing regions, then $\mathcal{D(\FF)}$ is planar.

\begin{cor}[of Raman and Ray \cite{RamanRay}]\label{cor:RR}
    If $\FF$ is a family of non-piercing regions, then the Delaunay graph $\mathcal{D}(\FF)$ is planar, therefore, it can have at most $3 |\FF|$ edges.
\end{cor}

Now we are ready to state the last lemma needed to complete the proof.

\begin{lemma}\label{lem:22pq}
    If $\FF$ is a family of non-piercing regions with the $(\nu+1,2)$-property, then $\FF$ also has the $(p,q)$-property for any $p>3e(\nu+1) \nu q+1$ for any $q\ge 2$, where $e=2.71...$ is Euler's number.
\end{lemma}

\begin{proof}
    The proof uses the so-called Clarkson-Shor method \cite{CS}.
    Fix some $q\ge 2$ and $p>3e(\nu+1)\nu q+1$, and suppose for a contradiction that there exists a subfamily $\GG \subset \FF $ of size $p$ without a common intersection. 
   Delete each set from $\GG$ with probability $1-\frac{1}{q}$ to obtain the subsubfamily $\HH$. The average number of remaining sets is $\frac{p}{q}$, so in expectation the Delaunay graph $\mathcal{D(\HH)}$ has this many vertices.
   The probability that two different intersecting sets of $\GG$ span an edge of $\mathcal{D(\HH)}$ is at least $\frac{1}{q^2} \left( 1 - \frac{1}{q} \right)^{q-3} > \frac{1}{e} \frac{1}{q^2}$, as any point in the intersection of two sets is contained in at most $q-3$ other sets; if these are all deleted, then we get a Delaunay edge.  As from any $\nu+1$ sets in $\GG$, there are two that intersect, the number of intersecting set pairs in $\GG$ is at least $\frac{\binom{p}{\nu+1}}{\binom{p-2}{\nu-1}} = \frac{p(p-1)}{(\nu+1)\nu}$.
   Consequently, the expected number of edges is at least $\frac{1}{e} \frac{p(p-1)}{(\nu+1)\nu q^2}$.
   As $\GG$ and all its subsystems are non-piercing, by Corollary \ref{cor:RR} the expected number of edges after deletion can be at most three times the expected number of vertices, that is, $\frac{1}{e} \frac{p(p-1)}{(\nu+1)\nu q^2} \leq 3\frac{p}{q}$, which contradicts $p>3e(\nu+1)\nu q +1$.
\end{proof}

 Lemmas \ref{lem:VCdim4} and \ref{lem:22pq} imply that $\FF$ satisfies the assumptions of Theorem \ref{pqthm} with $q=5$ and some large enough $p$, which implies that $\FF$ can be stabbed with a constant number of points.
 This finishes the proof of Theorem \ref{thm:pierce}. \qed

\section{Proof of Theorem \ref{thm:faces}}\label{sec:faces}

In this section, we present the exact statement and the proof of Theorem \ref{thm:faces}.

First, we introduce the definitions and notation following Axenovich, Kie{\ss}le and Sagdeev \cite{AKS}.
A plane graph is a graph that is embedded in the plane without crossing edges.
A 2-connected plane graph $G$ is \emph{$C_{< \ell}$-free} if it contains no cycle of length smaller than $\ell$.
$G$ is a  \emph{maximal $C_{< \ell}$-free plane graph} if adding any new edge would either create a crossing or a cycle of length less than $\ell$.
Define $f_{max}(\ell)$ as the largest possible face length of a 2-connected maximal $C_{< \ell}$-free plane graph.

Axenovich, Kie{\ss}le and Sagdeev \cite{AKS} showed that $f_{max}(\ell)=2\ell-3$ for $3 \leq \ell \leq 6$ using a former result of Axenovich, Ueckerdt, and Weiner \cite{AUW}.
For larger values of $\ell$, they showed a lower bound of $3\ell-9$ for $7\leq \ell \leq 9$, and $3\ell-12$ for $\ell \geq 10$.

They also showed an upper bound of $2(\ell-2)^2+1$ for any $\ell\geq 7$, and asked whether it could be improved to a linear upper bound. We give an affirmative answer to this question.

\begin{Th}[Theorem \ref{thm:faces}, restated]\label{thm:faces2}
    $f_{max}(\ell) \leq K \ell$ for some absolute constant $K$.
\end{Th}

We start with a simple observation made by Axenovich, Kie{\ss}le and Sagdeev \cite{AKS}.
 For two vertices $u,v$, let $d(u,v)$ denote their distance in $G$, and define the distance of a vertex $u$ and an edge $e=vw$ as $d(u,e)= \max\{d(u,v),d(u,w)\}$.
\begin{obs}[\cite{AKS}]\label{obs:distance}
    For any two vertices $u,v$ of a facial cycle of a maximal $C_{< \ell}$-free graph, their distance $d(u,v) \leq \ell-2$.
\end{obs}
\begin{proof}
    If $d(u,v) \geq\ell-1$, then $u$ and $v$ are non-adjacent, and the edge $uv$ could be added inside the cycle, preserving planarity, contradicting the maximality of our graph. 
\end{proof}

We first show that it is sufficient to prove Theorem \ref{thm:faces2} for the case when $\ell$ is even.
Suppose that $\ell$ is odd, and consider a plane graph $G$ with girth at least $\ell$ which is maximal with these constraints and has a face with boundary $C$ of length $m$. Subdivide each edge $e=uv$ with a middle vertex $e'$, thus creating a graph $G'$. It is easy to see that $G'$ is planar, and has girth at least $2\ell$. Additionally, it is maximal to these parameters:  By Observation \ref{obs:distance}, any two vertices on a face of $G$ are connected by a path of length at most $\ell-2$ 
(not necessarily along the face),
and every vertex of $G'$ either corresponds to a vertex of $G$ or is adjacent to one, therefore, any two vertices on a face of $G'$ are connected by a path of length at most $1+2(\ell-2)+1$. This implies that adding an edge between two vertices of $G'$ would either violate planarity or create a cycle of length at most $2\ell-1$.
Since $G'$ has a face of size $2m$, this implies that $2m \leq K \cdot  2\ell$, which implies our bound for $G$ as well.

Fix a maximal $\mathcal{C}_\ell$-free plane graph $G$ for some even $\ell$ and a facial cycle $C$ of $G$, we will bound the length $m$ of $C$ in terms of $\ell$. 
By Fáry's theorem, we can assume that $G$ is geometric, that is, each of its edges is a segment, without changing the topology of the embedding.

Let $\bar{B}(x,r)$ denote the closed (euclidean) disk of radius $r$ around a point $x$ in the plane, and for a simple curve $\gamma$, let $\bar{B}(\gamma,r)=\{x\in\mathbb R^2:\exists p\in \gamma\,\, |x-p|\leq r\}$, the set of points at euclidean distance at most $r$ from (a point of) $\gamma$. 

For each $v_i \in C$, we define a blow-up of its $(\ell /2-1)$-neighborhood in $G$.

Let $\rho$ be a small positive number that is less than half of the minimum euclidean distance of a vertex and a non-incident edge, and is also less than half of the minimum euclidean distance of two vertices.
Additionally, let $\delta <\rho$ be a small enough positive number such that for any two edges $e,f$ incident to a common vertex $u$, the intersection of $\bar{B}(e,\delta)$ and $\bar{B}(f,\delta)$ is contained  in $\bar{B}(u,\rho)$.
We also define $0<\varepsilon_1 < \varepsilon_2 < \dots < \varepsilon_m < \frac{\delta}{\ell^2}$ to be arbitrary positive numbers.

Finally, for $u \in V(G)$ and $i \in [m]$, let $\bar{B}^i (u)= \bar{B}\left(u,\frac{\rho}{d(v_i,u)+1} + \varepsilon_i \right) $, and for $e \in E(G)$ and $i \in [m]$, let $\bar{B}^i (e) = \bar{B}\left(e, \frac{\delta}{d(v_i,e) +1} + \varepsilon_i \right)$.

Now we can define the neighborhood regions $N_i$ for each vertex $v_i$ (see Figure \ref{fig:Ni}).

\[
    N_i= \bigcup\limits_{\substack{u \in V(G) \\ d(v_i,u)\leq \ell /2-1}}
    \bar{B}^i (u)
    ~~~\cup
    \bigcup\limits_{\substack{e \in E(G) \\d(v_i,e)\leq \ell /2-1}} 
    \bar{B}^i (e).
\]

\begin{figure}
\begin{center}
\begin{tikzpicture}[scale=0.8]
    \foreach \i in {0,1,2,...,15} {
        \node[inner sep=0pt] (v\i) at (22.5*\i:3cm) {};
    }
    
    \node[inner sep=0pt] (u1) at (-0.8,0.2) {};
    \node[inner sep=0pt] (u2) at (1.2,1) {};
    \node[inner sep=0pt] (u3) at (-1.3,-1.2) {};
    \node[inner sep=0pt] (u4) at (1.5,-1.2) {};
    \node[inner sep=0pt] (u5) at (0.2,-1) {};
    \node[inner sep=0pt](u6) at (1,-0.2) {};
    \node[inner sep=0pt] (u8) at (0.2,1.8) {};
    \node[inner sep=0pt] (u10) at (-2,0.5) {};

    \draw[line width=4 mm, color=red!30, opacity=0.9]  (v4) -- (v5);
    \draw[line width=4 mm, color=red!30, opacity=0.9]  (v4) -- (v3);
    \draw[line width=3 mm, color=red!30, opacity=0.9]  (v5) -- (v6);
    \draw[line width=3 mm, color=red!30, opacity=0.9]  (v3) -- (v2);
    \draw[line width=3 mm, color=red!30, opacity=0.9]  (v3) -- (u8);
    \draw[line width=2 mm, color=red!30, opacity=0.9]  (v6) -- (v7);
    \draw[line width=2 mm, color=red!30, opacity=0.9]  (u8) -- (u1);
    \draw[line width=2 mm, color=red!30, opacity=0.9]  (v2) -- (u2);
    \draw[line width=2 mm, color=red!30, opacity=0.9]  (v2) -- (v1);

    \filldraw[white] (v4) circle (5 mm);
    \filldraw[white] (v5) circle (4 mm);
    \filldraw[white] (v3) circle (4 mm);
    \filldraw[white] (v6) circle (3 mm);
    \filldraw[white] (v2) circle (3 mm);
    \filldraw[white] (u8) circle (3 mm);
    \filldraw[white] (v7) circle (2 mm);
    \filldraw[white] (v1) circle (2 mm);
    \filldraw[white] (u10) circle (2 mm);
    \filldraw[white] (u1) circle (2 mm);
    \filldraw[white] (u2) circle (2 mm);
    
    \filldraw[red!30, opacity=0.9] (v4) circle (5 mm);
    \filldraw[red!30, opacity=0.9] (v5) circle (4 mm);
    \filldraw[red!30, opacity=0.9] (v3) circle (4 mm);
    \filldraw[red!30, opacity=0.9] (v6) circle (3 mm);
    \filldraw[red!30, opacity=0.9] (v2) circle (3 mm);
    \filldraw[red!30, opacity=0.9] (u8) circle (3 mm);
    \filldraw[red!30, opacity=0.9] (v7) circle (2 mm);
    \filldraw[red!30, opacity=0.9] (v1) circle (2 mm);
    \filldraw[red!30, opacity=0.9] (u1) circle (2 mm);
    \filldraw[red!30, opacity=0.9] (u2) circle (2 mm);
    
    \draw (u4) -- (u5) -- (u3) -- (u1);
    \draw (u2) -- (u6) -- (u5);
    \draw (u8) -- (u1);
    \draw (u10) -- (u1);

    \draw (u2) -- (v2);
    \draw (u3) -- (v11);
    \draw (u8) -- (v3);
    \draw (u4) -- (v14);
    \draw (u10) -- (v7);

    \foreach \i in {0,1,2,...,15} {
        \node[circle,draw,scale=0.5,fill=black] at (v\i) {};
    }

    \node[circle,draw,fill=black,scale=0.5] at (u1)  {};
    \node[circle,draw,fill=black,scale=0.5] at (u2)  {};
    \node[circle,draw,fill=black,scale=0.5] at (u3){};
    \node[circle,draw,fill=black,scale=0.5] at (u4)  {};
    \node[circle,draw,fill=black,scale=0.5] at (u5) {};
    \node[circle,draw,fill=black,scale=0.5] at (u6)  {};
    \node[circle,draw,fill=black,scale=0.5] at (u8) {};
    \node[circle,draw,fill=black,scale=0.5] at (u10) {};

    \foreach \i in {0,1,2,...,15} {
        \pgfmathtruncatemacro{\j}{mod(\i+1,16)} 
        \draw (v\i) -- (v\j);
    }

    \node[above] at (v4) {\( v_i \)};
\end{tikzpicture}

\caption{ Illustration for a neighborhood $N_i$. 
}\label{fig:Ni}
\end{center}
\end{figure}
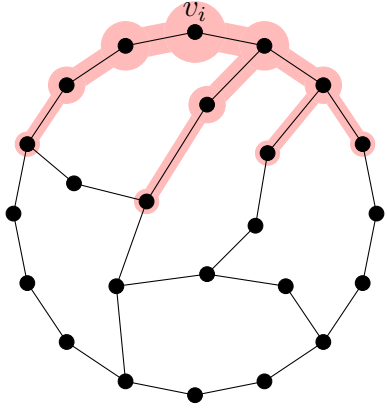

Note that the boundary of any two members of $\NN = \{N_1, N_2, \dots, N_m\}$ intersect in finitely many points as the radii of their blow-ups are perturbed with small amounts, so $\NN$ is a family of regions in general position.
In order to apply Theorem \ref{thm:pierce} to $\mathcal{N}$, we need to verify that it satisfies the properties required by the theorem.

\begin{prop}
    $N_i$ is simply connected for each $i$.
\end{prop}
\begin{proof}
    Since the girth of $G$ is at least $\ell$, the vertices whose distance from $v_i$ is at most $\ell /2-1$ form a tree.
\end{proof}

\begin{lemma}\label{lem:nonpiercing}
    The set-system $\mathcal{N}$ is non-piercing. 
\end{lemma}

While quite straight-forward, unfortunately our proof is quite tedious and technical.

\begin{proof}
Suppose for a contradiction that $N_i$ pierces $N_j$ for some $i\neq j$. Each point in $N_j\setminus N_i$ is contained in the neighborhood of a vertex or an edge which is closer to, or equally close to $v_j$ than to $v_i$.
Let $v_j'$ be an arbitrary point in $\bar{B}^j (v_j) \setminus N_i$.
We will show that any point $x \in N_j \setminus N_i$ is in the same connected component of $N_j \setminus N_i$ as $v_j'$, which contradicts our assumption.
Let $u_x$ be the vertex for which either 
$x \in \bar{B}^j(u_x)$,
or $x \in \bar{B}^j(e)$
for some edge $e$ incident to $u_x$ for which $d(v_j,e)= d(v_j, u_x) +1$. 
We will only discuss the case $\varepsilon_j < \varepsilon_i$, as the other case goes analogously.

Define $P_x=\{u_0=v_j,u_1, \dots, u_k=u_x\}$ to be the shortest path from $v_j$ to $u_x$
 in $G$. (This path is unique because $d(u_x,v_j)\le \ell /2-1$.)
 Note that for each vertex $w \in P_x$, $d(w,v_j) < d(w,v_i)$, where the inequality is strict because we assumed $\varepsilon_j <\varepsilon_i$. For an illustration, see Figure \ref{fig:connected}.

\begin{figure}
        \centering
        \begin{tikzpicture}
            \foreach \i in {0,1,2,...,3} {
                \node[inner sep=0pt] (v\i) at (1.5*\i,0) {};
            }
            \node[inner sep=0pt] (v4) at (6,0.5) {};

            \node[inner sep=0pt] (u0) at (4,1.5) {};
            \node[inner sep=0pt] (u1) at (3.5,-1.5) {};
            \node[inner sep=0pt] (u2) at (5,-1.5) {};
            \node[inner sep=0pt] (u3) at (1.5,1.8) {};

            \foreach \i in {0,1,2,...,3} {
            \pgfmathtruncatemacro{\nexti}{\i+1}
            \pgfmathsetmacro{\linewidth}{4 - \i*0.5}
            \draw[line width=\linewidth mm, color=red!30, opacity=0.9]  (v\i) -- (v\nexti);
            }
            \draw[line width=2.5 mm, color=red!30, opacity=0.9]  (u1) -- (v3);
            \draw[line width=2.5 mm, color=red!30, opacity=0.9]  (u2) -- (v3);
            \draw[line width=2.5 mm, color=red!30, opacity=0.9]  (u0) -- (v3);
            \draw[line width=2.5 mm, color=blue!30, opacity=0.9]  (u0) -- (u3);
            \draw[line width=2.8 mm, color=blue!30, opacity=0.9]  (u3) -- (0.8,2.8);
            \draw[line width=1.8 mm, color=red!30, opacity=0.9]  (u1) -- (2.8,-2);
            \draw[line width=1.8 mm, color=red!30, opacity=0.9]  (u2) -- (6,-2);
            \draw[line width=1.8 mm, color=red!30, opacity=0.9]  (v4) -- (7, 1);

            \filldraw[white, opacity=0.9] (v0) circle (6 mm);
            \filldraw[white, opacity=0.9] (v1) circle (5 mm);
            \filldraw[white, opacity=0.9] (v2) circle (4.2 mm);
            \filldraw[white, opacity=0.9] (v3) circle (3.5 mm);
            \filldraw[white, opacity=0.9] (v4) circle (3 mm);
            \filldraw[white, opacity=0.9] (u0) circle (3 mm);
            \filldraw[white, opacity=0.9] (u3) circle (3.5 mm);

            \filldraw[red!30, opacity=0.9] (v0) circle (6 mm);
            \filldraw[red!30, opacity=0.9] (v1) circle (5 mm);
            \filldraw[red!30, opacity=0.9] (v2) circle (4.2 mm);
            \filldraw[red!30, opacity=0.9] (v3) circle (3.5 mm);
            \filldraw[red!30, opacity=0.9] (v4) circle (3 mm);
            \filldraw[red!30, opacity=0.9] (u1) circle (3 mm);
            \filldraw[red!30, opacity=0.9] (u2) circle (3 mm);
            \filldraw[violet!50, opacity=0.9] (v3) circle (2.5 mm);
            \filldraw[violet!50, opacity=0.9] (v4) circle (2 mm);
            \filldraw[violet!50, opacity=0.9] (v2) circle (2 mm);
            \filldraw[violet!50, opacity=0.9] (u1) circle (2 mm);
            \filldraw[violet!50, opacity=0.9] (u2) circle (2 mm);
            \filldraw[blue!30, opacity=0.9] (u3) circle (3.8 mm);
            \filldraw[violet!50, opacity=0.9] (u3) circle (2 mm);
            \filldraw[blue!30, opacity=0.9] (u0) circle (3 mm);
            \filldraw[violet!50, opacity=0.9] (u0) circle (2.4 mm);


            \draw[line width=1.2 mm, color=violet!50, opacity=0.9]  (v2) -- (v3);
            \draw[line width=1.2 mm, color=violet!50, opacity=0.9]  (v4) -- (v3);
            \draw[line width=1.5 mm, color=violet!50, opacity=0.9]  (u0) -- (v3);
            \draw[line width=1.2 mm, color=violet!50, opacity=0.9]  (u0) -- (u3);
            \draw[line width=1.2 mm, color=violet!50, opacity=0.9]  (v3) -- (u1);
            \draw[line width=1.2 mm, color=violet!50, opacity=0.9]  (v3) -- (u2);

            \foreach \i in {0,1,2,...,4} {
                \node[circle,draw,scale=0.4,fill=black] at (v\i) {};
            }
        
            \node[circle,draw,fill=black,scale=0.4] at (u0)  {};
            \node[circle,draw,fill=black,scale=0.4] at (u1)  {};
            \node[circle,draw,fill=black,scale=0.4] at (u2)  {};
            \node[circle,draw,fill=black,scale=0.4] at (u3)  {};

            \foreach \i in {0,1,2,...,3} {
            \pgfmathtruncatemacro{\nexti}{\i+1}
            \draw (v\i) -- (v\nexti);
            }

           \draw (v3) -- (u0);
           \draw (v3) -- (u1);
           \draw (v3) -- (u2);
           \draw (u0) -- (u3);
           \draw (u3) -- (0.8,2.8);
           \draw (u1) -- (2.8,-2);
           \draw (u2) -- (6,-2);
           \draw (v4) -- (7, 1);

            \filldraw[red!30] (7.5,1) rectangle (8,1.5);
            \node at (8.8,1.2) {$ N_j\setminus N_i$};
            \filldraw[blue!30] (7.5,0) rectangle (8,0.5);
            \node at (8.8,0.2) {$  N_i\setminus N_j$};
            \filldraw[violet!50] (7.5,-1) rectangle (8,-0.5);
            \node at (8.8,-0.8) {$  N_j \cap N_i$};

            \node[above right] at (v0) {\( v_j \)};
            \node at ($(v2)+(4mm,4mm)$) {\( u_{k-1} \)};
            \node at ($(v3)+(3mm,4mm)$) {\( u_k \)};
        \end{tikzpicture}
        \caption{An illustration for intersections between different neighborhoods $N_i$ and  $N_j$.}
        \label{fig:connected}
    \end{figure}
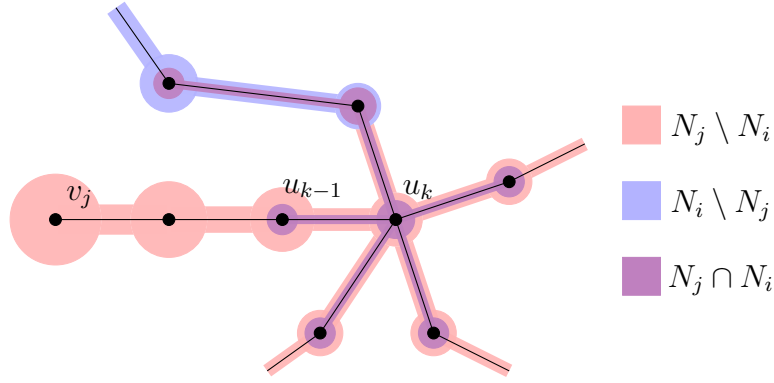

First, we show that points in the neighborhood of the same vertex are in the same connected component of $N_j \setminus N_i$, though to connect them we might need to leave the neighborhood of the respective vertex.

    \begin{prop}\label{prop:path}
        For any $u \in N_j$, the set $\bar{B}^j\left(u \right) \setminus N_i$ is contained in a connected component of $N_j \setminus N_i$.
    \end{prop}
    \begin{proof}
        We may assume $d(u,v_j) < d(u, v_i)$, as otherwise $\bar{B}^j\left(u \right) \setminus N_i$ is empty.
                
        We will use induction on $-d(u,v_i)$.

        If $d(u,v_i) \ge \ell /2$, then $u \notin N_i$, so we are done. 

        To prove the induction step, we need to show that any two points $p_1, p_2$ from $\bar{B}^j\left(u \right) \setminus N_i$ can be connected in $N_j \setminus N_i$. We may assume that $p_1$ and $p_2$ are on the boundary of $\bar{B}^j\left(u \right)$, as the segments connecting them to their respective projections from $u$ to the boundary are contained in $N_j \setminus N_i$. 

        We claim that there is at most one edge $e$ incident to $u$ for which the other endpoint $u_e$ of $e$ satisfies $d(u_e,v_i) \leq d(u,v_i)$. Consider such an edge $e$.
        In case of equality, there are two adjacent vertices in $G$ both at distance at most $\ell /2-1$ from $v_i$. The union of the two shortest paths and the edge between these two vertices creates a non-trivial closed walk of length at most $\ell /2+\ell /2-1$, which is a contradiction.
        Therefore, for each such edge, there is a shortest path going from $u$ to $v_i$ through $e$. If there were two such edges, then their union would create a non-trivial closed walk of length at most $\ell /2-1+\ell /2-1 < \ell $, which is again a contradiction.

        The points $p_1$ and $p_2$ divide the boundary circle $\partial\bar{B}^j\left(u \right)$ into two arcs. The above claim implies that for one of them, each edge $e$ crossing it satisfies $d(u,v_i) < d(u_e,v_i)$.

        This arc is divided into  subarcs by edges that cross it. We will show that each pair of subsequent subarcs are in the same connected component of $N_j \setminus N_i$. This will finish the proof of Proposition \ref{prop:path}.

        For each crossing edge $e$, both ends of $e$ are closer to $v_j$ than to $v_i$. This implies that the two boundary segments 
        $\partial  \bar{B}^j\left(e \right)\setminus \left( \bar{B}^j\left(u \right) \cup \bar{B}^j\left(u_e \right)\right)$ are contained in $N_j \setminus N_i$,
        and they connect subsequent arcs to the boundary circle of $\bar{B}^j\left(u_e\right)$. As  $d(u,v_i) < d(u_e,v_i)$, by induction, points on this boundary circle are in the same connected component of $N_j \setminus N_i$, and so the same is true for points of the subsequent arcs on $\partial \bar{B}^j\left(u\right)$.
    \end{proof}

    Now we return to the proof of Lemma \ref{lem:nonpiercing}.
    We will proceed by induction on the length of the path $P_x$, which is $k=d(u_k,v_j)$.
    
    The $k=0$ case is implied by Proposition \ref{prop:path}.

    To prove the induction step, it is sufficient to show that there exist two points $y_1 \in \bar{B}^j\left(u_k \right) \setminus N_i$ and $y_2 \in \bar{B}^j\left(u_{k-1}\right)\setminus N_i$  in the same connected component of $N_j \setminus N_i$, as Proposition \ref{prop:path} implies that points within the same disk are in the same connected component, and by induction, $y_2$ is in the same connected component as $v_j'$.
 
    Choose $y_1$ and $y_2$ from the neighborhood of the edge $e=u_{k-1}u_k$ such that $y_1 \in \partial \bar{B}^j \left(u_{k} \right) \cap \partial \bar{B}^j\left(e \right)$, $y_2 \in \partial \bar{B}^j\left(u_{k-1} \right) \cap \partial\bar{B}^j\left(e \right)$ and the segment $y_1y_2$ is on the boundary $\partial \bar{B}^j\left(e \right)$; see Figure \ref{fig:y1y2}.
    This way, the segment $y_1y_2\subset N_j \setminus N_i$ which proves that they are in the same connected component of $N_j \setminus N_i$. 

    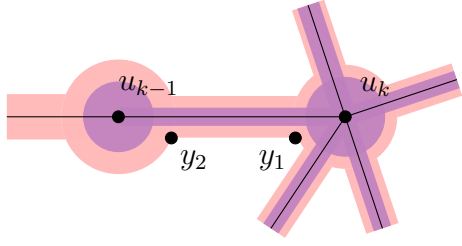
\begin{figure}
        \centering
        \begin{tikzpicture}
            \node[inner sep=0pt] (v1) at (0,0) {};
            \node[inner sep=0pt] (v2) at (1.5,0) {};
            \node[inner sep=0pt] (v3) at (4.5,0) {};
            \node[inner sep=0pt] (v4) at (6,0.5) {};
            \node[inner sep=0pt] (u0) at (4,1.5) {};
            \node[inner sep=0pt] (u1) at (3.5,-1.5) {};
            \node[inner sep=0pt] (u2) at (5,-1.5) {};

            \draw[line width=5.5 mm, color=red!30, opacity=0.9]  (v2) -- (v3);
            \draw[line width=4.5 mm, color=red!30, opacity=0.9]  (v3) -- (v4);
            \draw[line width=4.5 mm, color=red!30, opacity=0.9]  (u1) -- (v3);
            \draw[line width=4.5 mm, color=red!30, opacity=0.9]  (u2) -- (v3);
            \draw[line width=4.5 mm, color=red!30, opacity=0.9]  (u0) -- (v3);
            \draw[line width=6 mm, color=red!30, opacity=0.9]  (v1) -- (v2);

            \filldraw[white, opacity=0.9] (v3) circle (6.8 mm);
            \filldraw[white, opacity=0.9] (v2) circle (7.5 mm);

            \filldraw[red!30, opacity=0.9] (v3) circle (6.8 mm);
            \filldraw[red!30, opacity=0.9] (v2) circle (7.5 mm);
            \filldraw[violet!50, opacity=0.9] (v3) circle (5.2 mm);
            \filldraw[violet!50, opacity=0.9] (v2) circle (4.6 mm);


            \draw[line width=2.4 mm, color=violet!50, opacity=0.9]  (v2) -- (v3);
            \draw[line width=2.4 mm, color=violet!50, opacity=0.9]  (v4) -- (v3);
            \draw[line width=3 mm, color=violet!50, opacity=0.9]  (u0) -- (v3);
            \draw[line width=2.4mm, color=violet!50, opacity=0.9]  (v3) -- (u1);
            \draw[line width=2.4 mm, color=violet!50, opacity=0.9]  (v3) -- (u2);

            \foreach \i in {2,3} {
                \node[circle,draw,scale=0.4,fill=black] at (v\i) {};
            }
        
            \foreach \i in {1,2,3} {
            \pgfmathtruncatemacro{\nexti}{\i+1}
            \draw (v\i) -- (v\nexti);
            }

           \draw (v3) -- (u0);
           \draw (v3) -- (u1);
           \draw (v3) -- (u2);
           
            \node at ($(v3)+(4mm,4mm)$) {\( u_k \)};
            \node at ($(v2)+(4mm,4mm)$) {\( u_{k-1} \)};
            \node[circle,draw,scale=0.4,fill=black] (y2) at (2.2,-0.28) {};
            \node at ($(y2)+(3mm,-3mm)$) {\( y_2 \)};
            \node[circle,draw,scale=0.4,fill=black] (y1) at (3.84,-0.28) {};
            \node at ($(y1)+(-3mm,-3mm)$) {\( y_1 \)};

        \end{tikzpicture}
        \caption{Two points connected by a segment in $N_j \setminus N_i$.}
        \label{fig:y1y2}
    \end{figure}

This finishes the proof that $x$ and $v_j'$ are in the same component of $N_j \setminus N_i$ in the $i<j$ case.
The $i > j$ case of Lemma \ref{lem:nonpiercing} goes very similarly, so we omit the proof.
\end{proof}

\begin{lemma}
    The set-system $\mathcal{N}$ has the $(2,2)$-property, i.e., it is pairwise intersecting.
\end{lemma}
\begin{proof}
    Take two arbitrary vertices $v_i, v_j \in C$. By Observation \ref{obs:distance}, $d(v_i,v_j) \leq\ell-2$. Since $\ell$ is even, this implies that there is a vertex $w$ at distance at most $\ell /2-1$ from both $v_i$ and $v_j$. By the definition of $N_i$ and $N_j$, $w \in N_i \cap N_j$, which proves our statement.
\end{proof}

Now, Theorem \ref{thm:pierce} implies that there is a absolute constant $T$ such that there exists a point set $Z=\{z_1, z_2, \dots, z_T\}$ hitting each member of $\mathcal{N}$.

Note that we may take the points of $Z$ to be vertices of $G$, as the sets in $\mathcal{N}$ are unions of neighborhoods of vertices and edges of $G$, and if a set $N_i$ contains a point in the neighborhood of an edge $e=uv$, then $N_i$ contains both $u$ and $v$. As a consequence, we have a set of $T$ vertices such that for each vertex $v_i \in C$, there is an element of $Z$ at distance at most $\ell /2-1$ from $v_i$ in $G$.

Inspired by Axenovich, Kie{\ss}le and Sagdeev \cite{AKS}, we define a partitioning of $C$ into sets $C_1, C_2, \dots, C_T$ such that $v_i \in C_j$ if and only if $d(v_i,z_j) \leq d(v_i, z_{j'})$ for each $j' \neq j$, and $d(v_i,z_j) < d(v_i, z_{j'})$ for each $j' <j$. In other words, we assign each $v_i$ to the minimum index vertex in $Z$ which is closest to it.

Next, we present two statements about the distributions of the sets $C_i$, which help us bound the number of vertices on $C$.

The following observation was used by Axenovich, Kie{\ss}le and Sagdeev \cite{AKS} as well, we include its proof for completeness:
\begin{lemma}[\cite{AKS}]\label{lem:AKS}
    If  $v_{i+1}, \dots, v_{i+j} \in C_k $ are consecutive vertices of $C$ in the same partition class, then $j \leq\ell-1$. 
\end{lemma}
\begin{proof}
    Since for each $h$, the union of the shortest paths from $v_h$ to $z_k$ and from $v_{h+1}$ to $z_k$, and the edge $v_hv_{h+1}$ forms a closed walk of length at most $2(\ell /2-1)+1= \ell -1$, this walk has to be trivial, and thus $|(d(v_h,z_k)-d(v_{h+1},z_k)|=1$ must hold.

    Additionally, if there was an index $h$ such that $d(v_{h-1},z_k)+1= d(v_{h},z_k)  = d(v_{h+1},z_k)+1$, then there would be a non-trivial closed walk of length at most $2(\ell /2-1)=\ell-2$ through $v_h$ and $z_k$, which is a contradiction.

    As a corollary, the function $h \mapsto -d(v_h,z_k)$ is strictly unimodal.
    Since $0\le d(v_h,z_k)\le \ell /2-1$, it can take at most $2(\ell /2-1)+1=\ell-1$ different values, proving our result.
\end{proof}

\begin{lemma}\label{lem:DSsetup}
   There cannot exist indices $h_1 <h_2<h_3<h_4$, for which  $v_{h_1}, v_{h_3} \in C_i$ and $v_{h_2}, v_{h_4} \in C_j$ for some $i \neq j$.
\end{lemma}
\begin{proof}
    Assume without loss of generality that $i<j$.
    Let $P_a$  be a shortest path between $z_i$ and $v_{h_a}$ for $a \in \{1,3\}$, and similarly, let $P_b$ be a shortest path between $z_j$ and $v_{h_b}$ for $b\in \{2,4\}$.
    As $h_1 <h_2<h_3<h_4$, there must be an $a \in \{1,3\}$ and $b\in \{2,4\}$ for which $P_a$ and $P_b$ intersect in a vertex $w$. 
   If $d(w,z_i) \leq d(w,z_j)$, then $d(v_{h_b},z_i) \leq d(v_{h_b},z_j)$ which contradicts the definition of the sets $C_i,C_j$. Similarly, if $d(w,z_i) > d(w,z_j)$, then $v_{h_a}$ should be in $C_j$ rather than $C_i$.
\end{proof}

 We use a theorem of Davenport and Schinzel \cite{D-S}:
\begin{Th}[\cite{D-S}]\label{colorlemma}
     If we color the integers $\{1,2, \dots, n\}$ with $t$ colors such that  neighboring indices have different colors, and there are no indices $h_1 <h_2<h_3<h_4$ and colors $c_i, c_j$ such that $h_1, h_3$ have color $c_i$ and $h_2, h_4$ have color $c_j$, then $n \leq 2t-1$.
\end{Th}

Color vertices of $C_i$ with color $c_i$ for each $i$, and contract each monochromatic interval, thus getting an interval colored with $T$ colors which satisfies the conditions of Lemma \ref{colorlemma}. Each contracted interval had length at most $\ell -1$, therefore
the length of $C$ must satisfy $|C| \leq (2T-1)(\ell-1)$, which completes the proof of Theorem \ref{thm:faces2} and thus of Theorem \ref{thm:faces}. \qed

\newpage
\section{Face lengths on other surfaces}\label{sec:facesgen}

In this section, we show the following analogue of Theorem \ref{thm:faces} for graphs drawn on orientable surfaces:

\facesgen*

The proof will proceed similarly, with the main difference being the usage of a new result of Raman and Singh \cite{RS25}  on supports for subgraph systems on surfaces.

Instead of considering blow-ups of subgraphs and using their non-piercing property, we will consider the subgraphs themselves, with the non-piercing property exchanged to \emph{cross-freeness}, which was introduced by Raman and Singh \cite{RS25}:

\begin{Def}
For a fixed graph $H$ drawn on a surface and a pair of subgraphs $H_1,H_2$, 
the reduced graph $R(H_1,H_2)$ is the graph obtained by contracting all edges, both of
whose end-points are in $H_1 \cap H_2$.

Consider two subgraphs $H_1, H_2$ and a vertex $v$, and let $\tilde{v}$ be the image of $v$ in $R(H_1, H_2)$. 
Subgraphs $H_1$ and $H_2$ are said to be crossing at a vertex $v$ if there are edges $e_1,e_2,e_3,e_4 $ in $R(H_1,H_2)$ incident to $\tilde{v}$ in this circular order, with $e_1,e_3 \in H_1\setminus H_2$ and $e_2,e_4 \in H_2\setminus H_1$. \\
A set of subgraphs $\mathcal{H}=\{H_1, H_2, \dots, H_k \}$ without crossing pairs is called a \emph{cross-free} (or \emph{non-crossing}) graph system.
\end{Def}

We first show the following analogue of Theorem \ref{thm:pierce}:

\begin{Def}
    An $r$-neighbourhood around a vertex $u$ in a graph $G$ is the subgraph of $G$ induced by the vertices at distance at most $r$ from $u$ in $G$. 
\end{Def}

\begin{lemma}\label{lem:covergeneral}
    Let $G$ be a graph drawn on a surface with genus $g$ and girth larger than $2r+1$, and let $\mathcal{H}=\{H_1, H_2, \dots, H_k\}$ be a pairwise intersecting family of its $r$-neighbourhoods. Then $\mathcal{H}$ is a cross-free family and it can be pierced with at most $C(g)$ points for some $C(g)$ constant only depending on $g$.
\end{lemma}
\begin{proof}
We again show a bound for the dual VC-dimension and show a $(p,q)$-property for $\mathcal{H}$, and we conclude using Matou\v{s}ek's theorem.

\begin{obs}\label{obs:tree}
 $H_i$ is a tree for each $i$.
\end{obs}
\begin{proof}
 If $H_i$ contained a cycle, it would have length at most $2r+1$, contradicting the girth constraint. 
\end{proof}

\begin{lemma}
	The dual VC-dimension of $\mathcal{H}$ is less than $\frac{7+\sqrt{49+48g}}{2}$.
\end{lemma}
\begin{proof}
    Our proof starts similarly to Lemma \ref{lem:VCdim4}. Suppose, for a contradiction, that we have subgraphs $H_1, \dots, H_\varphi$ for $\varphi=\lceil\frac{7+\sqrt{49+48g}}{2}\rceil$ exhibiting this dimension, and choose points $x_i \in H_i \setminus \bigcup_{k\neq i} H_k$, $x_{i,j} \in \left(H_i\cap H_j\right) \setminus \bigcup_{k \neq i,j} H_k$. Take simple curves $\gamma_{i,j}$ to be paths from $x_i$ to $x_{i,j}$ inside $H_i$, and we again get an embedding of $K_\varphi$ on the surface.
    More precisely, to get a drawing where any two edges meet a finite number of times, we need to make a small perturbation to make the curves $\gamma_{i,j}$ edge-internally disjoint, i.e., any two curves can intersect only at a small neighborhood of a vertex.    
    Moreover, we can also assume that for any $i$ the curves $\{\gamma_{i,j}\}_j$ are pairwise disjoint (apart from all starting from $x_i$), as otherwise we could ``uncross'' such intersections.
    As the minimal genus of a surface into which $K_\varphi$ can be embedded is $\lceil\frac{(\varphi-3) (\varphi-4)}{12}\rceil > g $ \cite{RY68}, there will be two independent {intersecting} edges, 
    without loss of generality, $\gamma_{1,2}$ and $\gamma_{3,4}$.

    Suppose that $H_1$ and $H_3$ are neighbourhoods around $y_1$ and $y_3$, respectively, and let $z$ be a vertex in the small vicinity of an intersection point of $\gamma_{1,2}$ and $\gamma_{3,4}$. 

    By Observation \ref{obs:tree}, $H_1$ is a tree.
   Therefore, since the vertices $x_1, z, x_{1,2}$ lie on a path in this order, either $d(x_1,y_1)=d(x_1,z)+d(z,y_1)$ or $d(x_{1,2},y_1)=d(x_{1,2},z)+d(z,y_1)$. 
   Similarly, either $d(x_3,y_3)=d(x_3,z)+d(z,y_3)$ or $d(x_{3,4},y_3)=d(x_{3,4},z)+d(z,y_3)$.

    First, assume that the first equality holds both times.  
    By symmetry, we may assume $d(x_1,z) \geq d(x_3,z)$.
   Then,
    \[
    d(x_3,y_1) \le d(x_3,z)+d(z,y_1) \leq d(x_1,z)+d(z,y_1) = d(x_1,y_1),
    \]
    which implies $x_3 \in H_1$ as $x_1 \in H_1$, a contradiction since $x_3 \in H_3 \setminus H_1$.

     We also have $x_{1,2} \in H_1 \setminus H_3$ and  $x_{3,4} \in H_3 \setminus H_1$, so it does not matter which of the equalities hold, we could repeat the same proof with $x_{1,2}$ instead of $x_1$ and/or $x_{3,4}$ instead of $x_3$.
\end{proof}

\begin{prop}
 The subgraph-system $\mathcal{H}=\{H_1,H_2, \dots, H_k\}$ is cross-free.
\end{prop}
\begin{proof}
    Assume that two subgraphs, without loss of generality, $H_1$ and $H_2$, cross at a vertex $z$.  This implies the existence of vertices $x_1, x_2 \in H_1 \setminus H_2$ and $x'_1, x'_2 \in H_2 \setminus H_1$ such that $z$ lies on the unique path between $x_1$ and $x_2$ in $H_1$, and it also lies on the unique path between $x_1'$ and $x_2'$ in $H_2$,
    and we reach a contradiction the same way as in the proof of the previous lemma. 
\end{proof}

Recall that for a set-system $\FF$ on a set of elements $P$, we define the dual intersection hypergraph $\HH(\FF, P)$ such that the vertices correspond to members of $\FF$, while hyperedges correspond to elements of $P$ such that for every $p \in P$ the vertex set $H_p=\{F\in \FF: p \in F\}$ forms a hyperedge.
The Delaunay graph $\mathcal{D(\FF)}$ is the subgraph of $\HH(\FF, P)$ containing only the hyperedges with exactly two vertices---so each edge $F_1F_2$ is witnessed by some exclusive intersection point $x\in F_1\cap F_2$ such that no $F\in \FF\setminus \{F_1,F_2\}$ contains $x$.

Raman and Singh \cite{RS25} have shown that for any non-crossing graph system $\FF$ on a graph $G_\FF$ of genus $g$, there exists a \emph{dual support} of $\FF$ on genus at most $g$, that is, there exists a graph $G_\FF'$ of genus $g$ such that for each  $p \in V(G_\FF)$, $H_p$ is a connected subgraph of $G_\FF'$.

We have the following consequence for the Delaunay graph $\mathcal{D(\HH)}$, where each exclusive intersection point corresponds to an edge.

\begin{cor}[of Raman and Singh \cite{RS25} ]\label{cor:RS}
For any non-crossing graph system $\FF$ on a graph $G_\FF$ of genus $g$, 
 the Delaunay graph $\mathcal{D}(\mathcal{\FF})$ can be drawn on a surface with genus $g$, therefore it can have at most $3|F| - 6  + 6g$ edges.
\end{cor}

\begin{lemma}\label{lem:22pqgeneral}
    If $\FF$ is a family non-crossing subgraphs with the $(\nu+1,2)$-property, then $\FF$ also has the $(p,q)$-property for any $p>3e(\nu+1) \nu q+2gq +1$ for any $q\ge 2$, where $e=2.71...$ is Euler's number.
\end{lemma}

\begin{proof}
    The proof analogous to that of Lemma \ref{lem:22pq}, with the only difference being a different upper bound on the number of edges in the Delaunay graph.
    
    Fix some $q\ge 2$ and $p>3e(\nu+1)\nu q+1$, and suppose for a contradiction that there exists a subfamily $\GG \subset \FF $ of size $p$ without a common intersection. 
   Delete each set from $\GG$ with probability $1-\frac{1}{q}$ to obtain the subsubfamily $\HH$. The average number of remaining sets is $\frac{p}{q}$, so in expectation the Delaunay graph $\mathcal{D(\HH)}$ has this many vertices.
   The probability that two different intersecting sets of $\GG$ span an edge of $\mathcal{D(\HH)}$ is at least $\frac{1}{q^2} \left( 1 - \frac{1}{q} \right)^{q-3} > \frac{1}{e} \frac{1}{q^2}$, as any point in the intersection of two sets is contained in at most $q-3$ other sets; if these are all deleted, then we get a Delaunay edge.  As from any $\nu+1$ sets in $\GG$, there are two that intersect, the number of intersecting set pairs in $\GG$ is at least $\frac{\binom{p}{\nu+1}}{\binom{p-2}{\nu-1}} = \frac{p(p-1)}{(\nu+1)\nu}$.
   Consequently, the expected number of edges is at least $\frac{1}{e} \frac{p(p-1)}{(\nu+1)\nu q^2}$.
   As $\GG$ and all its subsystems are non-crossing, by Corollary \ref{cor:RS}, we get the following inequality for the expected number of edges after deletion: $\frac{1}{e} \frac{p(p-1)}{(\nu+1)\nu q^2} \leq 3\frac{p}{q} + 6g$. 
   Reordering, we get $p(p - (3e (\nu +1) \nu q+1)) \leq 6eg (\nu +1) \nu q^2$, which contradicts $p>3e(\nu+1) \nu q+2gq +1$. 
\end{proof}

As in the proof of Theorem \ref{thm:pierce}, we conclude using Theorem \ref{pqthm}.
\end{proof}

\begin{proof}[Proof of Theorem \ref{thm:facesgen}]
Fix a graph $G$ with girth at least $\ell$ drawn on a surface with genus $g$, and fix a face of $G$ with facial cycle $C=v_1v_2 \dots v_m$. Similarly as in the proof of Theorem \ref{thm:faces}, we may assume that $l$ is even.

For every vertex $v_i \in C$, define its $\left(\frac{\ell}{2}-1\right)$-neighbourhood $N_i=\{\ v \in V(G): d(v,v_i)\leq \frac{\ell}{2}-1 \}\cup\{e \in E(G): d(e,v_i) \leq \frac{\ell}{2} -1\}$.

	\begin{lemma}
	The subgraph-system $\mathcal{N}$ has the $(2,2)$-property, i.e., it is pairwise intersecting.
\end{lemma}
\begin{proof}
	Take two arbitrary vertices $v_i, v_j \in C$. By Observation \ref{obs:distance}, $d(v_i,v_j) \leq\ell-2$. Since $\ell$ is even, this implies that there is a vertex $w$ at distance at most $\ell /2-1$ from both $v_i$ and $v_j$. By the definition of $N_i$ and $N_j$, $w \in N_i \cap N_j$, which proves our statement.
\end{proof}

Now, we can use Lemma \ref{lem:covergeneral} with $r = \frac{\ell}{2}-1$, getting $C(g)$ points $z_1, z_2, \dots, z_{C(g)}$ covering all neighbourhoods $N_i$.

Again, we define a partioning of the vertices of the cycle such that $v_i \in C_j$ if and only if $d(v_i,z_j) \leq d(v_i, z_{j'})$ for each $j' \neq j$, and $d(v_i,z_j) < d(v_i, z_{j'})$ for each $j' <j$.

As Lemma \ref{lem:AKS} does not use any topological properties of the surface, it still holds here. In contrast, we need to show a substitute for Lemma \ref{lem:DSsetup}: 
\begin{lemma}\label{lem:DSsetupgeneral}
	There cannot exist indices $i_1 <j_1<i_2<j_2<\dots<j_b$, for which  $v_{i_1}, v_{i_2}, \dots, v_{i_b} \in C_i$ and $v_{j_1}, v_{j_2}, \dots, v_{j_b} \in C_j$ for some $i \neq j$, where $b=4g+3$.
\end{lemma}
\begin{proof}
	Assume that there exist two such sets $C_i,C_j$. We will show that this implies the the existence of a proper embedding of $K_{3,b}$ on the vertex set $\{z_i,z_j,z\}\cup\{v_{i_i}, v_{i_2}, \dots, v_{i_b}\}$, where $z$ is an arbitrary point of the empty face. 
	
	For $h \in [b]$, let the edge $v_{i_h}z_i$ of the embedding be the shortest path between $v_{i_h}$ and $z_i$ in the graph, and let the edge $v_{i_h}z_j$ of the embedding be concatenation of the shortest path between $v_{j_h}$ and $z_i$, and the path between $v_{i_h}$ and $v_{j_h}$ on $C$.
    Note that each path from $z_i$ is disjoint from each path from $z_j$, and the paths from $z_i$ (respectively, from $z_j$) can be made internally pairwise disjoint with a small perturbation.    
    Finally, $z$ can be connected to vertices $v_{i_h}$ with arbitrary, pairwise non-intersecting curves. 
	
	It is easy to see that these curves yield a proper embedding of a $K_{3,b}$ graph on the surface, which implies $b \leq 4g+2$, as the genus of the complete bipartite graph $K_{p,q}$ is $\lceil \frac{(p-2)(q-2)}{4} \rceil$ \cite{R1965}.
	
\end{proof}

Here, we use the general form of the Davenport-Schinzel theorem \cite{D-S}: Let $\lambda_b(t) = t \cdot 2^{\frac{1}{s!} \alpha(t)^s+O\left(\alpha(t)^{s-1}\right)}$, where $s=\lfloor \frac{b-2}{2} \rfloor$ and $\alpha(t)$ is the inverse Ackermann function.
\begin{Th}[\cite{D-S}]
     If we color the integers $\{1,2, \dots, n\}$ with $t$ colors such that  neighboring indices have different colors, and there are no indices $h_1 <h_2<\dots < h_{2b-1} <h_{2b}$ and colors $c_i, c_j$ such that $h_1, h_3, \dots, h_{2b-1}$ have color $c_i$ and $h_2, h_4, \dots, h_{2b}$ have color $c_j$, then $n \leq \lambda_b(t)$.
\end{Th}

Again, we color vertices of $C_i$ with color $c_i$ for each $1 \leq i \leq C(g)$, and contract each monochromatic interval, thus getting an interval colored with $C(g)$ colors satisfying the conditions of Lemma \ref{lem:DSsetupgeneral}. By Lemma \ref{lem:AKS}, each contracted interval had length at most $\ell-1$, therefore $|C| \leq \lambda_{4g+3}(C(g)) (\ell-1)$.
\end{proof}

\subsection*{Concluding remarks and open questions}

There are several questions left open, from which we would like to highlight a few.

\begin{quest}\label{quest1}
    Is it true that every non-piercing pairwise intersecting family can be pierced with 4 points?
    What about pseudo-disks?
\end{quest}

If yes, the method of our proof for Theorem \ref{thm:faces} would give an upper bound of $7\ell-7$ for $f_{max}(\ell)$, improving the current best bound $8\ell-13$ of Axenovich, Kie{\ss}le, Sagdeev, and Zhukovskii \cite{AKS}.

\begin{quest}
    What is the exact constant for $f_{max}(\ell)?$
    Is it possible to improve the constructions in \cite{AKS}?
    What is the exact characterization of maximal $\mathcal{C}_{<\ell}$-free graphs?
\end{quest}

\subsubsection*{Acknowledgement}

We would like to thank Arsenii Sagdeev for discussion, in particular, to pointing out that a positive answer to Question \ref{quest1} would give a better upper bound for $f_{max}(\ell)$.
He also pointed out that the methods of their paper and ours might be combined to obtain an improved upper bound, but we leave this for future research.

\newpage

\nocite{*}
\bibliographystyle{plain}
\bibliography{template}

\begin{thebibliography}{10}

\bibitem{ANPPSS}
P.~K. Agarwal, E.~Nevo, J.~Pach, R.~Pinchasi, M.~Sharir, and S.~Smorodinsky.
\newblock {Lenses in arrangements of pseudo-circles and their applications}.
\newblock {\em Journal of the ACM}, 51(2):139--186, 2004.

\bibitem{AK92}
N.~Alon and D.~Kleitman.
\newblock {Piercing convex sets and the Hadwiger--Debrunner (p, q)-problem}.
\newblock {\em Advances in Mathematics}, 96:103--112, 1992.

\bibitem{AKS}
M.~Axenovich, L.~Kießle, and A.~Sagdeev.
\newblock {Faces of maximal plane graphs without short cycles}, 2024.
\newblock arXiv:2410.13481v1.

\bibitem{AUW}
M.~Axenovich, T.~Ueckerdt, and P.~Weiner.
\newblock {Splitting planar graphs of girth 6 into two linear forests with short paths}.
\newblock {\em Journal of Graph Theory}, 85(3):601--618, 2017.

\bibitem{BPR}
S.~Buzaglo, R.~Pinchasi, and G.~Rote.
\newblock {\em {Topological Hypergraphs}}, pages 71--81.
\newblock Springer New York, New York, NY, 2013.

\bibitem{CKM}
P.~Carmi, M.~J. Katz, and P.~Morin.
\newblock {Stabbing Pairwise Intersecting Disks by Four Points}.
\newblock {\em Discrete and Computational Geometry}, 70:1751--1784, 2023.

\bibitem{CS}
K.~L. Clarkson and P.~W. Shor.
\newblock {Applications of random sampling in computational geometry, II}.
\newblock {\em Discrete and Computational Geometry}, 4:387--421, 1989.

\bibitem{DGRR}
Suryendu Dalal, Rahul Gangopadhyay, Rajiv Raman, and Saurabh Ray.
\newblock Sweeping arrangements of non-piercing regions in the plane.
\newblock In Wolfgang Mulzer and Jeff~M. Phillips, editors, {\em 40th International Symposium on Computational Geometry, SoCG 2024, June 11-14, 2024, Athens, Greece}, volume 293 of {\em LIPIcs}, pages 45:1--45:15. Schloss Dagstuhl - Leibniz-Zentrum f{\"{u}}r Informatik, 2024.

\bibitem{Danzer}
L.~Danzer.
\newblock {Zur L{\"o}sung des Gallaischen Problems {\"u}ber Kreisscheiben in der Euklidischen Ebene}.
\newblock {\em Studia Scientiarum Mathematicarum Hungarica}, 21(1--2):111--134, 1986.

\bibitem{D-S}
H.~Davenport and A.~Schinzel.
\newblock {A Combinatorial Problem Connected with Differential Equations}.
\newblock {\em American Journal of Mathematics}, 87(3):684--694, 1965.

\bibitem{DSW}
Guo-Li Ding, Paul Seymour, and Peter Winkler.
\newblock Bounding the vertex cover number of a hypergraph.
\newblock {\em Combinatorica}, 14(1):23--34, 1994.

\bibitem{DJ12}
A.~Dumitrescu and M.~Jiang.
\newblock {Coloring translates and homothets of a convex body}.
\newblock {\em Beitrage zur Algebra und Geometrie}, 53(2):365--377, 2012.

\bibitem{HD57}
H.~Hadwiger and H.~Debrunner.
\newblock {{\"U}ber eine Variante zum Helly’schen Satz}.
\newblock {\em Archiv der Mathematik}, 8:309--313, 1957.

\bibitem{handbook}
A.~Holmsen and R.~Wenger.
\newblock {Helly-type theorems and geometric transversals}.
\newblock In {\em Handbook of Discrete and Computational Geometry}. CRC Press, Boca Raton, FL, 2017.

\bibitem{HQRX25}
Xinqi Huang, Yuzhen Qi, Mingyuan Rong, and Zixiang Xu.
\newblock {Largest dyadic dual VC-dimension of non-piercing families}, 2025.
\newblock arXiv:2506.13606.

\bibitem{KellerSmorodinsky25}
Chaya Keller and Shakhar Smorodinsky.
\newblock {A simple proof of a (p,2)-theorem for non-piercing regions}, 2025.
\newblock arXiv:2507.07269.

\bibitem{Matousek}
J.~Matou{\v s}ek.
\newblock {Bounded VC-Dimension Implies a Fractional Helly Theorem}.
\newblock {\em Discrete and Computational Geometry}, 31(2):251--255, 2004.

\bibitem{rom}
Rom Pinchasi.
\newblock A finite family of pseudodiscs must include a ``small'' pseudodisc.
\newblock {\em SIAM J. Discrete Math.}, 28(4):1930--1934, 2014.

\bibitem{RamanRay}
R.~Raman and S.~Ray.
\newblock {Planar Support for Non-piercing Regions and Applications}.
\newblock In {\em 26th Annual European Symposium on Algorithms}, volume 112 of {\em LIPIcs}, pages 69:1--69:14, 2018.

\bibitem{RS24}
R.~Raman and K.~Singh.
\newblock {On Hypergraph Supports}, 2024.
\newblock arXiv:2303.16515.

\bibitem{RS25}
R.~Raman and K.~Singh.
\newblock {On Supports for graphs of bounded genus}, 2025.
\newblock arXiv:2503.21287.

\bibitem{RY68}
G.~Ringel and J.~W.~T. Youngs.
\newblock {Solution of the Heawood map-coloring problem}.
\newblock {\em Proceedings of the National Academy of Sciences of the United States of America}, 60(2):438--445, 1968.

\bibitem{R1965}
Gerhard Ringel.
\newblock Das {G}eschlecht des vollständigen paaren {G}raphen.
\newblock {\em Abhandlungen aus dem Mathematischen Seminar der Universität Hamburg}, 28:139--150, 1965.

\bibitem{SS20}
A.~Scott and P.~Seymour.
\newblock {A survey of {{\(\chi\)}}-boundedness}.
\newblock {\em Journal of Graph Theory}, 95(3):473--504, 2020.

\bibitem{Stacho}
L.~Stach{\'o}.
\newblock {A Gallai-f{\'e}le k{\"o}rlet{\"u}z{\'e}si probl{\'e}ma megold{\'a}sa}.
\newblock {\em Matematikai Lapok}, 32:19--47, 1981--84.

\bibitem{Tutte}
W.~T. Tutte.
\newblock {Toward a theory of crossing numbers}.
\newblock {\em Journal of Combinatorial Theory}, 8:45--53, 1970.

\end{thebibliography}

\end{document}